\documentclass[11pt]{amsart}
\usepackage{amsfonts}
\usepackage{amsfonts,latexsym,rawfonts,amsmath,amssymb,amsthm}
\usepackage[plainpages=false]{hyperref}

\usepackage{graphicx}

\numberwithin{equation}{section}

\newcommand{\beq}{\begin{equation}}
\newcommand{\eeq}{\end{equation}}
\newcommand{\beqs}{\begin{eqnarray*}}
\newcommand{\eeqs}{\end{eqnarray*}}
\newcommand{\beqn}{\begin{eqnarray}}
\newcommand{\eeqn}{\end{eqnarray}}
\newcommand{\beqa}{\begin{array}}
\newcommand{\eeqa}{\end{array}}

\def\lra{\longrightarrow}

\def\bc{\begin{center}}
\def\ec{\end{center}}

\def\begeq{\begin{equation}}
\def\endeq{\end{equation}}
\def\and{\quad{\rm and}\quad}

\let\lra=\longrightarrow

\def\mapright\#1{\,\smash{\mathop{\lra}\limits^{\#1}}\,}

\newtheorem{prop}{Proposition}[section]
\newtheorem{theo}[prop]{Theorem}
\newtheorem{lem}[prop]{Lemma}

\newtheorem{cor}[prop]{Corollary}
\newtheorem{rem}[prop]{Remark}

\newtheorem{defi}[prop]{Definition}

\title  {Convergence of K\"ahler-Ricci flow on Fano manifolds, II}

\author   {Gang Tian}
\author { Xiaohua $\text{Zhu}^*$}

\thanks {* Partially supported by the NSFC Grant 10990013}

\address {Gang Tian\\ School of Mathematical Sciences and BICMR, Peking
University, Beijing, 100871, China\\ and Department of Mathematics, Princeton
University,  New Jersey,  NJ 02139, USA\\
 tian@math.mit.edu}
\address{ Xiaohua Zhu\\School of Mathematical Sciences and BICMR, Peking University,
Beijing, 100871, China\\
 xhzhu@math.pku.edu.cn}

\begin{document}
\bibliographystyle{plain}

\begin{abstract} In this paper, we give an alternative proof for the convergence of  K\"ahler-Ricci flow on a Fano mnaifold $(M,J)$. This proof
differs from that in [TZ3].  Moreover,  we generalize the main theorem of [TZ3] to the case that $(M,J)$ may not admit any K\"ahler-Einstein metrics.

\end{abstract}

\maketitle

\section { Introduction}

In this paper, we use Perelman's $W$-functional in [Pe] to study the
convergence of  K\"ahler-Ricci flow on a Fano mnaifold. We  want to
prove

\begin{theo}
\label{th-1} Let $(M,J)$ be  a compact K\"ahler-Einstein manifold
with positive first Chern class $c_1(M)>0$. Then for any initial
K\"ahler metric $g$  with  K\"ahler class $2\pi c_1(M)$,
K\"ahler-Ricci flow converges to a K\"ahler-Einstein metric in the
$C^\infty$-topology.  Moreover, the convergence can be made
exponentially modulo a holomorphism transformation on $(M,J)$.
\end{theo}

Theorem \ref{th-1} was first announced by G. Perelman when he
was visiting MIT in the Spring of 2003. In [TZ3], we gave a proof of
Theorem \ref{th-1} by using an inequality of Moser-Trudinger type
established in [T1] and [TZ1]\footnote{We need to add more details about
how to use the Moser-Trudinger typed inequality in the case that $M$ has non-trivial
holomorphic fields, .}. The main purpose
of this paper is to give an alternative proof of this theorem. This
new proof does not use the inequality of Moser-Trudinger type.
Moreover, our proof also yields a generalization as follows:

\begin{theo}
\label{th-2} Let $(M,J)$ be  a compact K\"ahler manifold with
positive first Chern class $ c_1(M)>0$. Suppose that there exists  a
K\"ahler metric $g_0$  with K\"ahler class $2\pi c_1(M)$ such that
K\"ahler-Ricci flow  with the initial metric $g_0$ converges to a
K\"ahler-Einstein metric $g_{KE}$ in the $C^\infty$-topology, which
may not be compatible with $J$. Then  for any initial K\"ahler
metric $g$  with  K\"ahler class $ 2\pi c_1(M)$, the K\"ahler-Ricci
flow converges to $g_{KE}$  in the $C^\infty$-topology.
\end{theo}

Theorem \ref{th-2} implies that the convergence of K\"ahler-Ricci
flow is independent of the choice of initial K\"ahler metrics with
K\"ahler class $ 2\pi c_1(M)$. The main idea in the proof of Theorem
1.1 and Theorem 1.2 comes from [TZ4], where we used Perelman's
$W$-functional to study the stability of K\"ahler-Ricci flow on a
Fano manifold. Recently Sun and Wang obtained a more general result
for the stability [SW]. Our proof depends on certain asymptotic
estimates for the minimizing functions (we call $f_t$-functions)
defined by the $W$-functional associated to evolved metrics $g_t$ of
K\"ahler-Ricci flow. Our estimates depend only on the lower bound of
Mabuchi's $K$-energy on the space of K\"ahler potentials on $(M,J)$
[Ma]. If $(M,J)$ is a  K\"ahler-Einstein manifold, Bando and Mabuchi
proved that the K-energy is bounded from below ([BM], also see
[DT]). Recently, in [Ch], Chen gives a generalization of this. In
fact, Chen's result will be used in the proof of Theorem 1.2.

The organization of this paper is as follows: In Section 2, we
outline an approach of proof to Theorem 1.1. In Section 3, we recall
the W-functional and discuss some of its applications to
K\"ahler-Ricci flow (cf. Proposition 3.3 and Lemma 3.4). In Section
4, we give some  estimates for  $f_t$-functions along the
K\"ahler-Ricci flow. Theorem 1.1 and Theorem 1.2 will be proved in
Section 6 and Section 7, respectively. In Section 5, a stability
result for K\"ahler-Ricci flow, Proposition 2.1, will be proved.
Section 8 is an appendix in which we give a $C^0$-estimate for
$f_t$-functions.

\noindent {\bf Acknowledgements.} Part of the  work for this paper was done when both authors were visiting the Simons Center for Geometry and Physics
at Stony Brook  in early 2011. The authors would like to thank people there for hospitality and excellent working condition.

\section {An approach of continuity method to K\"ahler-Ricci flow}

Let $(M,J)$ be an $n$-dimensional compact K\"ahler manifold with
positive first Chern class $c_1(M)>0$ and $g$ be a K\"ahler metric
with its K\"ahler class equal to $2\pi c_1(M)$.  We consider the
(normalized) K\"ahler-Ricci flow (abbreviated as KR-flow):
\begin{equation}\label{eq:krf1}
\frac{\partial g(t)}{\partial t}\,=\,-{\rm Ric}(g(t)) +
g(t),~~~g(0)=g.
\end{equation}
It was proved in [Ca] that \eqref{eq:krf1} has a global solution
$g_t=g(t)$ for all time $t>0$.
For simplicity, we denote by $(g_t;g)$ a solution of \eqref{eq:krf1} with initial metric $g$.

Now we assume that $(M,J)$ admits a K\"ahler-Einstein metric
$\omega_{KE}$ with K\"ahler class $2\pi c_1(M)$.

\begin{prop}
\label{prop-2.1} Let  $(M,J)$ be  a K\"ahler-Einstein
manifold with positive first Chern class $c_1(M)>0$ and $g_{KE}$ be
a K\"ahler-Einstein  metric  on $(M,J)$ with its K\"ahler form
$\omega_{g_{KE}}\in 2\pi c_1(M)$. Then there exists a number
$\delta>0$ which depends only on $g_{KE}$ such that for any K\"ahler metric $g$ on $(M,J)$
with K\"ahler form $\omega_g \in 2\pi c_1(M)$ satisfying
 \begin{align} \|g-g_{KE}\|_{C^3(M)}\le \delta,
 \end{align}
KR-flow $(g_t;g)$ converges to $g_{KE}$  globally in the $C^\infty$-topology.
\end{prop}

Proposition \ref{prop-2.1} is a special case of more general
stability theorem on K\"ahler-Ricci flow on a Fano manifold, which
was recently proved by Sun and Wang by using an inequality of
Lojasiewicz type for the K\"ahler-Ricci flow [SW]. We note that
Proposition \ref{prop-2.1} was also proved by the second named
author if the condition (2.2) is replaced by certain condition on
K\"ahler potentials [Zhu]. In this paper, for the readers'
convenience, we will include a proof of Proposition \ref{prop-2.1}
in Section 5.  Similar argument will be used in proving Theorem 1.2.
Here we denote by $\|g-g'\|_{C^{\ell}(M)}$ the   norm:
$$\|g-g'\|_{C^{\ell}(M)}\,=\,\inf_{\Phi}\{ |g-\Phi^*(g')|_{C^{\ell}(M)}\},$$
where $\Phi$ runs over all diffeomorphisms of $M$ and
$|g_1-g_2|_{C^{\ell}(M)}$ denotes the $C^\ell$-norm between two
tensors $g_1$ and $g_2$ on $M$. The convergence of Riemannian
metrics in the $C^\infty$-topology means the convergence in all
$C^{\ell}$-norms $\|\cdot\|_{C^{\ell}(M)}$.

 We can write
$$\omega_g=\omega_\phi=\omega_{g_{KE}}+\sqrt{-1}\partial\overline\partial\phi\in 2\pi c_1(M)$$
for a K\"ahler potential $\phi$ on $(M,J)$ and we define a path of
K\"ahler forms
$$\omega_s=\omega_{g_{KE}}+s\sqrt{-1}\partial\overline\partial\phi.$$
Set
\begin{align} I=\{& s \in ~[0,1]~|~ (g_t^s; \omega_s)~\text{converges to}~g_{KE}~\text{in the}~C^\infty-\text{topology} \}.\notag
\end{align}
Clearly, it follows from Proposition \ref{prop-2.1} that $I$ is not empty.
We want to show that $I$ is both open and closed. It follows that $I=[0,1]$. Theorem 1.1 will follow from this.

The openness of $I$ follows from the following corollary of Proposition \ref{prop-2.1}:

\begin{cor}
\label{cor-1} Let $(M,J,\omega_{g_{KE}})$ be a K\"ahler-Einstein
manifold with $\omega_{g_{KE}}\in 2\pi c_1(M)>0$. Suppose that there
exists a K\"ahler metric $g_0$ on $(M,J)$ with K\"ahler class $2\pi
c_1(M)$ such that KR-flow $(g_t;g_0)$ converges to  $g_{KE}$ in the
$C^\infty$-topology. Then there exists a $\delta>0$ which depends
only on $g_{KE}$ and $ g_0$ such that for any K\"ahler metric $g$ on
$(M,J)$ with K\"ahler class $2\pi c_1(M)$ satisfying
\begin{align} \|g-g_0\|_{C^{3}(M)}\le \delta,
\end{align}
KR-flow $(g_t;g)$ converges to $g_{KE}$ in the $C^\infty$-topology.
\end{cor}

\begin{proof} By the assumption, we can choose $T$ sufficiently large such that
if $\tilde g_t$ is the solution of \eqref{eq:krf1} with initial metric $g_0$, then
$$\| \tilde g_T-g_{KE}\|_{C^{3}(M)}< \frac{\delta}{2},$$
where $\delta$ is a small number determined in Proposition \ref{prop-2.1}.
Since the K\"ahler-Ricci flow is stable for any fixed finite time, there is a small $\epsilon >0$ such that whenever
$\| g-g_0\|_{C^{3}(M)} < \epsilon$, we have
$$\| g_T - \tilde g_T\|_{C^{3}(M)}\,<\, \frac{\delta}{2}.$$
Hence, we have
\begin{align}\| g_T - g_{KE}\|_{C^3(M)}\,<\, \delta.\end{align}
Then the flow $(g_t; g_T)$ with initial $g_T$ will
converge to $g_{KE}$ in the $C^\infty$-topology according to Proposition \ref{prop-2.1}. This proves the corollary.
\end{proof}

It remains to prove

\begin{theo}\label{th-3}
$I$ is closed.
\end{theo}
This will be proved in Section 6.

\section{Perelman's W-functional}

In this section, we review  Perelman's $W$-functional in [Pe]. The
$W$-functional is defined for triples $(g,f,\tau)$ on a given closed
manifold $M$ of dimension $m$, where $g$ is a Riemannian metric, $f$
is a smooth function and $\tau$ is a constant. It is defined as
follows:
\begin{align} W(g,f,\tau)=(4\pi\tau)^{-m}\int_M[\tau(R(g)+ |\nabla f|^2)+f-m]
e^{-f}dV_g,\end{align} where $R(g)$ denotes the scalar curvature of
$g$. We also normalize the triple $(g,f,\tau)$ by
\begin{equation}
\label{norm-1} (4\pi\tau)^{-m}\int_M e^{-f} dV_g\,=\, (4\pi\tau)^{-m} V.
\end{equation}
In our situation, since $M$ is a compact K\"ahler manifold of complex dimension $n$,
we further have the following normalization of the volume of $g$:
\begin{equation}
\label{vol-1} \int_M dV_g \,=\,(2\pi)^n\int_M c_1(M)^n\,\equiv\, V.
\end{equation}
Note that $m=2n$. Then the $W$-functional depends only on a pair $(g,f)$ and can be
reexpressed as follows:
\begin{align}
\label{w-fun-1}
W(g,f)=&\,W(g,f,\frac{1}{2}) + (2\pi)^{-m}m V \,\notag\\
=&\, (2\pi)^{-m}\int_M[\frac{1}{2}(R(g)+ |\nabla f|^2)+f]
e^{-f}dV_g,\end{align} where $(g,f)$ satisfies \eqref{norm-1} with
$\tau= \frac{1}{2}$ and \eqref{vol-1}. Note that \eqref{vol-1} holds automatically for certain geometric
problems, for example, metrics which evolve along a normalized Ricci
flow.

Following Perelman, we define an entropy $\lambda(g)$ by
$$\lambda(g)=\inf_f \{W(g,f)|~ f~\text{satisfies}~\eqref{norm-1}\}.$$
It is well-known that $\lambda(g)$ can be attained by some $f$ (cf.
[Ro]). In fact, such a $f$ satisfies the Euler-Lagrange equation of
$W(g,\cdot)$,
\begin{equation}
\label{w-equ-1} \triangle f + f + \frac{1}{2} (R - |\nabla
f|^2)\,=\,\lambda(g).
\end{equation}
Following a computation in [Pe], we can deduce the first
variation\footnote{It is not obvious that $\lambda(\cdot)$ is
differentiable. This can be proved by using the comparison principle
and the regularity results we have on \eqref{w-equ-1} (see Lemma 3.5
below and Appendix  in Section 8). Here we just compute its first
variation at those $g$ where $\lambda $ is smooth. This is the case
when the minimizer $f$ is unique. It is sufficient for us in this
paper.} of $\lambda(\cdot)$,
\begin{equation}
\label{w-equ-2} \delta\lambda(g)\,=\,-\,\int_M <\delta g,
\text{Ric}(g)-g+ \nabla^2f> e^{-f} dV_g,
\end{equation}
where $\text{Ric}(g)$ denotes the Ricci tensor of $g$ and $\nabla^2f$ is
the Hessian of $f$. It follows from \eqref{w-equ-2} that $g$ is a
critical point of $\lambda(\cdot)$ if and only if $g$ is a gradient
shrinking Ricci soliton, namely, $g$ satisfies
\begin{equation}
\label{w-equ-3}\text{Ric}(g)\, + \,\nabla^2f \,= \,g,
\end{equation}
where $f$ is a minimizer of $W(g,\cdot)$.

Next we will prove the uniqueness of solutions for \eqref{w-equ-1}
if $g$ is a gradient shrinking Ricci soliton.

\begin{lem}
\label{unique-1} If  $g$ satisfies \eqref{w-equ-3}, then any solution of
\eqref{w-equ-1} is equal to the function $f$ in \eqref{w-equ-3}
modulo a constant. Consequently, the minimizer of $W(g,\cdot)$ is
unique if the metric $g$ is a gradient shrinking Ricci soliton.
\end{lem}

\begin{proof} Let $\sigma_t$ be an one-parameter subgroup generated
by the gradient field
$$X=\frac{1}{2}\nabla_g f=\frac{1}{2}g^{ij}f_j\frac{\partial}{\partial x_i}.$$
Put $\tilde g(t,\cdot)\,=\,\sigma_t^\star g$. Then $\tilde
g(t,\cdot)$ is a family of Ricci solitons and $\tilde g$ satisfies
the normalized Ricci flow,
$$\frac{\partial \tilde g}{\partial t}\,=\,L_{(\sigma_t^{-1})_\star X}\tilde g\,=\,-\text{Ric}(\tilde g)+\tilde g.$$

Let $f'$ be another solution of \eqref{w-equ-1} and $\rho_t$ be an
one-parameter subgroup generated by the gradient field
$$X'\,=\,\frac{1}{2} \nabla_{\tilde g} f'.$$
Set $g'\,=\,\rho_t^\star \tilde g (t,\cdot)$. Then $ g'$ satisfies a
modified Ricci flow:
\begin{align}
\frac{\partial g'}{\partial t}&\,=\,-\text{Ric}( g')+
g' - L_{(\rho_t^{-1})_\star X'} g'\notag\\
&\,=\,-\,\text{Ric}( g')+ g' - \nabla^2 \rho_t^\star f'.\notag\end{align}
Clearly, $(\sigma_t\circ\rho_t)^\star f'$ satisfies \eqref{w-equ-1}
with $\hat g_t\,:=\,(\sigma_t\circ\rho_t)^\star g$ in place of $g$.
Therefore, using the first variation of the functional $W$ and its
invariance under diffeomorphisms, we have
\begin{align} 0&\,=\,\frac{d W((\sigma_t\circ\rho_t)^\star g, (\sigma_t\circ\rho_t)^\star f')} {dt}\notag\\
&\,=\,-\int_M<\frac{ d \hat g_t }{dt}, \text{Ric}(\hat g_t)- \hat
g_t + \nabla^2 (\sigma_t\circ\rho_t)^\star f' >\,e^{-
(\sigma_t\circ\rho_t)^\star f' }
dV_{\hat g_t}\notag\\
&\,=\,\int_M | \text{Ric}(\hat g_t)- \hat g_t + \nabla^2
(\sigma_t\circ\rho_t)^\star f' |^2\,e^{-
(\sigma_t\circ\rho_t)^\star f'} dV_{\hat g_t}.\notag
\end{align}
It follows
$$\text{Ric}(g)- g+\nabla^2 f'=0.$$
Hence, $f'\,=\,f +\text{const.}$.
\end{proof}

As a converse to Lemma \ref{unique-1}, we have

\begin{lem}
\label{unique-2} Let $g$ be a gradient shrinking Ricci soliton and
$f$ be the function in \eqref{w-equ-3} for $g$. Then
$f$ satisfies \eqref{w-equ-1}.
\end{lem}

\begin{proof} Differentiating \eqref{w-equ-3} and applying the second
Bianchi identity, we have
$$\nabla_iR\,=\,2 R_{ij}f_j.$$
It follows
$$\nabla_iR+2f_{ij}f_j-2f_i\,=\,2(R_{ij}+f_{ij}-g_{ij})f_j\,=\,0.$$
This implies
$$R+|\nabla f|^2-2f\equiv const,$$
and consequently, by using \eqref{w-equ-3}, we deduce that $f$ is a
solution of \eqref{w-equ-1}.
\end{proof}

\begin{prop}
\label{unique-3} Let $(M, J,g_{KE})$ be a K\"ahler-Einstein
manifold with K\"ahler class $2\pi c_1(M)$. Then $g_{KE}$ is a
global maximizer of $\lambda(\cdot)$ in the space of K\"ahler
metrics $(g,J')$ on $M$ with K\"ahler class $2\pi c_1(M)$.
\end{prop}

\begin{proof} By Lemma \ref{unique-1}, the
minimizer $f$ of $W(g_{KE},\cdot)$ must be a constant. As a
consequence, $f=0$ because of the
normalization condition \eqref{norm-1}. Thus for any K\"ahler metric
$g$ with K\"ahler class $2\pi c_1(M)$, we
have
\begin{align} \lambda(g_{KE})&\,=\,W(g_{KE}, 0)\,=\,W(g,0)\notag\\
&\,\ge\,\inf \,\{ W(g,f) ~|~\int_M e^{-f} dV_g \,=\,V\}
\,= \,\lambda(g)\notag.
\end{align}
Here we have used the fact:
\begin{align} &W(g,0)=\,n (2\pi)^{-2n}
 \,\int_M \text{Ric} (g)\wedge \omega_{g}^{n-1}\notag\\
 &=\,n(2\pi)^{-2n} \,\int_M \text{Ric} (g_{KE})\wedge \omega_{g_{KE}}^{n-1}\,=\,n (2\pi)^{-2n} V.
 \notag
\end{align}

\end{proof}

The following is a direct corollary of Proposition 3.3 by
using the K\"ahler-Ricci flow.

\begin {lem} Let $(J_0,g_{KE})$ and $ (J, g)$ be a
K\"ahler-Einstein metric and a K\"ahler metric on $M$ with $2\pi c_1(M)$ as both
K\"ahler classes. Suppose that
$\lambda(g)\ge\lambda(g_{KE})$. Then $g$ is a K\"ahler-Einstein
metric with respect to the complex structure $J$.
\end{lem}

\begin {proof} Let $g(t,\cdot)$ be a solution of \eqref{eq:krf1} with the initial metric $g$. Then by the
monotonicity of $\lambda(g)$ along the Ricci flow [Pe] ( see (3.6)
or \eqref{eq:lambda-3} below), we have
$$\lambda(g(t,\cdot))\ge \lambda(g),~\forall~ t>0.$$
Thus by Proposition 3.3, we see
$$\lambda(g(t,\cdot))=\lambda(g_{KE}),~\forall~ t>0.$$
Again by (3.6), it follows that $g(t,\cdot)$ are all gradient
shrinking Ricci solitons. Moreover,  the K\"ahler condition implies
that $g(t,\cdot)$ are in fact all K\"ahler-Ricci solitons. Therefore,
by the uniqueness of K\"ahler-Ricci solitons [TZ2], we see that $g(t,\cdot)$ are
all the same modulo automorphisms of $M$, and consequently, $g$ is a K\"ahler-Ricci soliton with
respect to some holomorphic field $X$ induced by a Hamiltonian
function $f$ which satisfies (3.5). On the other hand,
$$\lambda(g)=W(g,f)=W(g,0)=\lambda(g_{KE}).$$
This implies that $f=0$ since the
minimizer of $W(g,\cdot)$ is unique according to Lemma 3.1. Hence, $g$ is a
K\"ahler-Einstein metric with respect to the complex structure $J$.

\end{proof}

We will end this section by a well-known fact, which follows from the monotonicity of
Perelman's entropy $\lambda(\cdot)$.
For the readers' convenience, we include its proof.

\begin{lem} Let $g_t=g(t,\cdot)$ be a solution of \eqref{eq:krf1} on $M$. Suppose that there exists a sequence $g_i$ of $g_t$
converging to a limit Riemannian metric $g_\infty$ in the $C^{3}$-topology. Then $g_\infty$ is a
K\"ahler-Ricci soliton with respect to some complex structure $J_\infty$.

\end{lem}

\begin{proof} By the assumption, there exists a sequence of diffeomorphisms $\Psi_i$ of $M$
such that
$$ |\Psi_i^\star
g_i-g_\infty|_{C^{3}(M)}\to 0.$$ It follows that after taking a
subsequence if necessary, $(\Psi_i^{-1})_\star \cdot
J\cdot(\Psi_i)_{\star} $ converges to an integral complex structure
$J_\infty$ in the $C^{3}$-topology. The standard regularity theory
implies that $J_\infty$ is in fact analytic.  By (3.6), we have
\begin{align}
\label{eq:lambda-3} \frac{d}{dt}\lambda(g_t)\,=\,\int_M
|\sigma_t^\star(\text{Ric}( g_t)-g_t +\nabla^2 f_t)|_{\sigma_t^\star
g}^2 dV_{\sigma_t^\star g},
\end{align}
where $f_t$ are minimizing solutions of (3.5) associated to metrics $g_t$ and
$\sigma_t$ is the family of diffeomorphisms of $M$ generated by the time-dependent vector field
$\frac{1}{2}\nabla_{g_t} f_t$. Since
$$\lambda(g_t)\le W(g_t,0)\,=\,(2\pi)^{-m}\frac{m}{2} V$$
is bounded from above, there exists a sequence of $\alpha_i$
such that as $i$ tends to $\infty$, $|\alpha_i - i|\to 0$ and
\begin{align}\label{eq:conv-2}
\int_M |\text{Ric}( g_{\alpha_i})-g_{\alpha_i} +\nabla^2
f_{\alpha_i}|_{g_{\alpha_i}}^2 dV_{g_{\alpha_i}}\to 0.
\end{align}

Since the Ricci flow depends continuously on initial metrics, we may choose $\delta_0 > 0$
such that
\begin{align}
\label{eq:conv-1} |\Psi_i^\star g_{\alpha_i}-g_\infty
|_{C^{3}(M)}\,\le\,  \delta_0
\end{align}
as $i\to\infty$. On the other hand,  by Proposition 8.1 in our
Appendix, we know that $f_t$ are uniformly bounded. Thus by applying
the regularity theory of elliptic equations to (8.1) in Appendix and
using \eqref{eq:conv-1}, we get
$$\|\Psi_{i}^\star f_{\alpha_i}\|_{C^{4,\frac{1}{2}}(M)} \,\le\, C$$
for some uniform constant $C$. Hence, by taking a subsequence if necessary, there exists a
$C^{4,\frac{1}{2}}$-smooth function $f_\infty$ such that
$$\Psi_{\alpha_i}^\star f_{\alpha_i}\to f_\infty$$
in the $C^{4}$-topology. Then it follows from \eqref{eq:conv-2}
$$\text{Ric}( g_\infty)-g_\infty +\nabla^2 f_\infty\,=\,0.$$
By the regularity theory, $f_\infty$ is actually smooth. Furthermore, since
$g_\infty$ is K\"ahler, the gradient
vector field $\nabla f_\infty$ has to be holomorphic. This proves that
$g_\infty$ is a K\"ahler-Ricci soliton.

\end{proof}

\section{Estimates for $f$-function }

In this section, we derive some a priori estimates for the $f_t$-functions solving (3.5) associated to the KR-flow (2.1).
Since the flow preserves the K\"ahler class, we may write the
K\"ahler form of $g_t$ as
$$\omega_{\phi}\,=\,\omega_g+\sqrt{-1}\partial\overline\partial \phi$$ for some K\"aher potential $\phi=\phi_t$.
Furthermore, as usual, one can choose $\phi_t$ appropriately such that
(2.1) is reduced to a parabolic complex Monge-Amp\`ere flow
for $\phi(t,\cdot)=\phi_t$,
\begin{align}\label{eq:krf2}
\frac {\partial\phi}{\partial t}=\log \frac {\det (g_{i\overline
j}+\phi_{i\overline j})}{\det (g_{i\overline j})}+\phi-h',~~
  \phi (0,\cdot)=0,\end{align}
where $h'=-\frac{\partial\phi}{\partial t}|_{t=0}$ is a Ricci
potential of the initial metric $g$. It is easy to check that each
$\frac{\partial\phi}{\partial t}$ is a Ricci potential of $g_t$. We
set $h_t=-\frac{\partial\phi}{\partial t} + c_t$ for some constant
$c_t$ so that
\begin{align} \label{eq:ricci-1}
\int_M e^{h_t}\omega_{g_t}^n\,=\,\int_M\omega_g^n.\end{align}

The following estimates are due to G. Perelman. We refer the readers
to [ST] for their proof.

\begin{lem}\label{lem:perelman-1}  There are constants $c$ and $C$ depending only on the initial metric $g$ such that (a)
$\text{diam}(M,g_t)\le C$; (b) $\text{vol}(B_r(p),g_t )\ge c r^{2n}$; (c) $ \|h_{t}\|_{C^0(M)}\le C$;
(d) $\|\nabla h_{t}\|_{g_t} \le C$; (e) $\|\Delta h_{t} \|_{C^0(M)} \le C$.
 \end{lem}

Under the assumption that the K-energy is bounded from below, we can improve the
above estimates as follows.

\begin{prop} If in addition the K-energy is bounded from
below on the space of K\"ahler potentials on $(M,J)$ in Lemma 4.1. Then we have:

\noindent
(a) $\lim_{t\to \infty}\| h_{t}\|_{C^0(M)}=0$;

\noindent
(b) $\lim_{t\to \infty}\|\nabla h_{t}\|_{g_t}=0$;

\noindent
(c)  $\lim_{t\to \infty}\|\Delta h_{t} \|_{C^0(M)}=0$.
\end{prop}

\begin{proof} Recall that the K-energy $\mu(\cdot)$ is defined for a K\"ahler potential $\phi$ on $(M,J)$ by
$$\mu(\phi)=-\int_0^1\int_M \dot\psi (R(\psi_s)-n)\omega_{\psi_s}^n,$$
where $\psi_s$ is a path of K\"ahler potentials on $(M,J)$  and
$R(\psi_s)$ denote scalar curvatures of K\"ahler forms
$\omega_{\psi_s}$. Then for a family of K\"ahler potentials
$\phi=\phi_t$ in (4.1), we have
$$\frac{d\mu(\phi)}{dt}=-\int_M |\nabla h_{t}|_{g_t}^2\omega_{g_t}^n.$$
Thus by the lower bound of $\mu(\cdot)$, one sees that there exists a
sequence of $t_i\in [i,i+1]$ such that
\begin{align}\lim_{t_i\to \infty} H(t_i)= \lim_{t_i\to \infty}\int_M |\nabla h_{t_i}|_{g_{t_i}}^2\omega_{g_{t_i}}^n=0.\end{align}
Since $H(t)$ satisfies a differential equation,
$$\frac{d H(t)}{dt}\le CH(t),$$
for some uniform constant (cf. [CT2], [PSSW]),
it is easy to derive
\begin{align} \lim_{t\to \infty}\int_M |\nabla h_{t}|_{g_{t}}^2\omega_{g_{t}}^n=0.\end{align}

Let
$$\tilde h_t\,=\,h_t-\frac{1}{V}\int_M h_te^{h_t}\omega_{g_t}^n.$$
By using the weighted Poincare inequality in [TZ3], we have
$$\int_M \tilde h_t^2  e^{h_t}\omega_{g_{t}}^n\le \int_M |\nabla h_{t}
|_{g_t}^2 e^{h_t}\omega_{g_{t}}^n.$$
It follows from (c) in Lemma 4.1,
\begin{align} \lim_{t\to\infty} \int_M \tilde h_t^2 \omega_{g_{t}}^n=0.\end{align}
We claim
\begin{align} \lim_{t\to\infty} \|\tilde h_t\|_{C^0(M)} =0.\end{align}
So we get (a) in our proposition by the normalization condition (4.2).
In fact the claim follows  from an inequality
\begin{align}\|\tilde h_t\|_{C^0(M)}^{n+1}\le C \|\nabla h_t\|_{g_t}^{n}[\int_M \tilde h_t^2 \omega_{g_{t}}^n]^{\frac{1}{2}}.\end{align}
(4.7) can be proved by using the non-collapsing estimate (b) in Lemma 4.1 (cf. [PSSW], [Zh]). By (d) in Lemma 4.1, the claim follows.

Next, set
$$a_t=-\int_M
\frac{\partial\phi}{\partial t}\omega_{g_t}^n,$$
a direct computation shows
$$\frac {d a_t}{dt}=a_t-\int_M |\nabla h_t|_{g_t}^2 \omega_{g_t}^n.$$
It follows that
$$e^{-t} a_t - a_0 \,=\,\int_0^t e^{-s} |\nabla h_s|_{g_s}^2\omega_{g_s}^n\wedge ds.$$
On the other hand, since the K-energy is bounded from below, as in [CT1] or [TZ3], we can normalize $h'$  in (4.1)
by adding a suitable constant so that
$$\int_M h'\omega_g^n
=-\int^{\infty}_0\int_M |\nabla h_t|_{g_t}^2
e^{-t}\omega_{g_t}^n\wedge dt.$$
Thus we get
\begin{align} a_t \,=\,\int_t^{\infty} e^{t-s} |\nabla h_s|_{g_s}^2\omega_{g_s}^n\wedge ds\,\to\, 0,~\text{as}~t\to\infty.\end{align}
Hence, by (a) in our proposition,  we also get
\begin{align} \lim_{t\to\infty}\|\frac{\partial\phi}{\partial t}\|_{C^0(M)}=0.\end{align}
Now (b) and (c) follow from the standard gradient and Laplacian
estimates for the K\"ahler-Ricci flow which can be derived by using
the Maximum principle as in [Ba], [T1]. For instance, the following
lemma was proved in [PSSW] and can be applied to complete the proof
of (b) and (c).

\end{proof}

\begin{lem}([PSSW]) There exist $\delta, K>0$ depending only on the dimension $n$ with the following property: For any
  $\epsilon>0$ with $\epsilon<\delta$ and any $t>0$, if
  $$\|\frac{\partial\phi}{\partial t}\|_{C^0(M)}\le \epsilon,$$
then
$$\|\nabla h_{t+2}\|_{g_{t+2}}+\|\Delta h_{t+2}\|_{C^{0}(M)}\le K\epsilon.$$
\end{lem}

From Proposition 4.2, we can deduce

\begin{prop} Suppose that the K-energy is bounded from
below. Then there exists a sequence of $t_i\in [i,i+1]$ such that

\noindent
(a) $\lim_{t_i\to \infty}\|\Delta f_{t_i} \|_{L^2(M,\omega_{g_{t_i}})}\,=\,0$;

\noindent
(b) $\lim_{t_i\to \infty}\|\nabla f_{t_i}\|_{L^2(M,\omega_{g_{t_i})}}\,=\,0$;

\noindent
(c) $\lim_{t_i\to \infty} \int_M f_{t_i} e^{-f_{t_i}} \omega_{g_{t_i}}^n \,=\, 0$.

\end{prop}

\begin{proof}
Let $\sigma_t$ be the family of diffeomorphisms generated by the time-dependent gradient vector field
$\frac{1}{2}\nabla_{g_t} f_t$. Then by (3.6), we have
\begin{align} \frac{d}{dt}\lambda(g_t)&=\int_M \sigma_t^\star(|\text{Ric}( g_t)-g_t +\nabla^2 f_t|^2e^{-f_t}) [\sigma_t^\star(\omega_{g_t})]^n,\notag\\
&=\int_M |\text{Ric}( g_t)-g_t +\nabla^2 f_t|_{g_t}^2
e^{-f_t}\omega_{g_t}^n.\notag
\end{align}
It follows from (2.1) that
\begin{align} \frac{d}{dt}\lambda(g_t)
&=\int_M |\frac{1}{2}\partial\overline\partial h_t -D^2
f_t|_{g_t}^2 e^{-f_t}\omega_{g_t}^n\notag\\
&\ge \frac{1}{2n}\int_M |\Delta(f_t-h_t)|^2 e^{-f_t}\omega_{g_t}^n.\notag
\end{align}
Since $\lambda(g_t)\le (2\pi)^{-n}nV$ are uniformly bounded, we see that there exists a sequence of $t_i\in [i,i+1]$ such that
$$\lim_i\int_M |\triangle(f_{t_i}-h_{t_i})|^2 e^{-f_{t_i}}\omega_{g_{t_i}}^n=0.$$
Thus by Proposition 4.2,  we will get (a).  Here we used the fact that $f_t$ are uniformly bounded according to Proposition 8.1 in Appendix.
Moreover, by
\begin{align} \int_M |\nabla f_{t}|^2 \omega_{g_t}^n\,= \,-\int_M f_{t} \triangle f_{t}  \omega_{g_t}^n\,\le\, C \left ( \int_M
|\Delta f_{t}|^2  \omega_{g_t}^n\right )^{\frac{1}{2}}\notag, \end{align}
Thus we get (b) from (a).

Define
$$\tilde f_t=f_t-\frac{1}{V}\int_Mf_t e^{h_t}\omega_{t}^n.$$
By using the weighted Poincare inequality in [TZ3], we have
$$\int_M \tilde f_t^2 e^{h_t} \omega_{g_t}^n\le 2\int_M |\nabla f_{t}|^2 e^{h_t}\omega_{g_t}^n\le  C [\int_M|\Delta f_{t}
|^2 \omega_{g_t}^n]^{\frac{1}{2}}.$$
Hence, by using (a) and the boundedness of $h_t$, we can deduce
\begin{align} \label{eq:f-t1}
\int_M \tilde f_{t_i}^2 \omega_{g_{t_i}}^n\to 0,~\text{as}~t_i\to\infty.
\end{align}
Since
$$-\triangle \tilde f_t=-\Delta f_t= f +\frac{1}{2} (R - |\nabla
f|^2)-\lambda(g)\le C',$$
we have
$$-\Delta \tilde f_t^+\le C',$$
where $\tilde f_t^+=\max\{\tilde f_t, 0\}$.
By using the standard Moser's iteration, we can derive
$$\tilde f_t^+\le C'' \|\tilde
f_t^+\|_{L^2(M,g_t)}\le C''\|\tilde f_t\|_{L^2(M,g_t)}.$$
Thus by \eqref{eq:f-t1}, we get
\begin{align}\lim_{t_i\to \infty} \|\tilde f_{t_i}^{+}\|_{C^0(M)}=0.\end{align}

Let $\tilde f_t^-=\min\{\tilde f_t, 0\}$. Since $f_t$ is uniformly bounded and $h_t$ converges to $0$, it follows from  $\int_M \tilde f_t
e^{h_t}\omega_{g_t}^n=0$
\begin{align} \lim_{t_i\to \infty} \int_M \tilde f_{t_i}^{-} \omega_{g_{t_i}}^n = 0.\notag
\end{align}
Thus
\begin{align}|\int_M \tilde f_{t_i} e^{-\tilde f_{t_i}}\omega_{g_{t_i}}^n|&\le \int_{M^+} \tilde f_{t_i}^+ e^{-\tilde f_{t_i}}\omega_{g_{t_i}}^n
+\int_{M^{-}}(-\tilde f_{t_i}^- ) e^{-\tilde f_{t_i}} \omega_{g_{t_i}}^n\notag\\
&\le\|\tilde f_{t_i}^{+}\|_{C^0(M)} \int_M e^{-\tilde f_{t_i}} \omega_{g_{t_i}}^n
+  \sup_M e^{-\tilde f_{t_i}}\int_{M}(-\tilde f_{t_i}^-) \omega_{g_{t_i}}^n\notag\\
&\to 0,~\text{as}~t_i\to\infty.
\end{align}
Here we denote $M^{+}=\{\tilde f_{t_i}\ge 0\}$ and  $M^{-}=\{\tilde f_{t_i}\le 0\}$ and we used the fact  that $\tilde f$ are uniformly bounded.
By (4.12) and the normalization condition (3.2), we will get the property (c) since $f_t$ differs from  $\tilde f_t$ by a function $c(t)$ with
$\lim_{i\to \infty} c(t_i)= 0$.

\end{proof}

\begin{defi} The energy level $L(g)$ of entropy $\lambda(\cdot)$ along the KR-flow $(g_t; g)$
with an initial K\"ahler metric $g$ is given by
$$L(g)\,=\,\lim_{t\to\infty}\lambda(g_t).$$

\end{defi}

By the monotonicity of $\lambda(g_t)$, we see that $L(g)$ exists and it is finite. In case that the K-energy is bounded from
below, from Proposition 4.4 together with Proposition 4.2, we get

\begin{cor}Suppose that the Mabuchi's K-energy is bounded from
below. Then
 $$L(g)\,=\,(2\pi)^{-n}nV\,=\,\sup\{\lambda(g')|~\omega_{g'}\in 2\pi c_1(M)\}.$$
\end{cor}

The above corollary shows that the energy level $L(g)$ of entropy $\lambda(\cdot)$
does not depend on the initial K\"ahler metric $g$.

\section {Proof of Proposition \ref{prop-2.1}}

The proof of Proposition \ref{prop-2.1} depends on the following uniqueness result recently proved by Chen and Song in [CS].
This more general uniqueness has been conjectured by the first named author in [T2]. Note that
the uniqueness of K\"ahler-Einstein metrics on a Fano manifold was proved by Bando and Mabuchi in 1985 for
a fixed complex structure [BM].

\begin{theo}([CS])  Let $(M,J)$ be a compact K\"ahler manifold with $c_1(M)>0$.
Suppose that there are two sequences of K\"ahler metrics with K\"ahler class $2\pi c_1(M)$, which converge to K\"ahler-Einstein metrics
$(g_1, J_1)$ and $(g_2, J_2)$ in the $C^\infty$-topology, respectively.  Then $(g_1, J_1)$ is conjugate to $(g_2, J_2)$ by a diffeomorphism of $M$.
\end{theo}

\begin{rem} In [CS], Chen and Sun also showed a version of Theorem 5.1 for K\"ahler metrics with constant scalar curvature and
integral K\"ahler class.
\end{rem}

\begin{rem}
In fact, we only need a much weaker version of Theorem 5.1 in order to prove Proposition 2.1. Such a weaker version can be proved by using the deformation theory.
\end{rem}

\begin{proof}[Proof of Proposition \ref{prop-2.1}]
Let $g_t=g(t,\cdot)$ be a family of evolved K\"ahler metrics of
KR-flow $(g_t;g)$. We claim that there exists $\delta>0$ such that if
$g$ satisfies (2.2), then
\begin{align}\|g_t-g_{KE}\|_{C^{3}(M)}\le \epsilon(\delta),~\forall~
t>0, \end{align}
where $\epsilon(\delta)\to 0$ as $\delta\to 0.$  We  will use an  argument by
contradiction from [TZ4].
On contrary, we may find a sequence of K\"ahler metrics
$g^i$ on $(M,J)$ with K\"ahler forms in $2\pi c_1(M)$
and a sequence of diffeomorphisms $\Psi_i$ on $M$
which satisfy
\begin{align} |\Psi_i^\star
g^i-g_{KE}|_{C^{3}(M)}\to 0, \end{align} such that there exists a
number $\epsilon_0>0$ and a sequence of K\"ahler metrics
$g^i(t_i,\cdot)$ with property:
\begin{align} \|g^i_{t_i}-g_{KE}\|_{C^{3}(M)}\ge
\epsilon_0,\end{align}
where $g^i_{t_i}$ is a solution of KR-flow $(g^i_t; g^i)$ at time $t_i$.
Moreover, by the stability of Ricci flow in  finite time and
by the continuity of the Ricci flow, one can choose the number $\epsilon_0$ arbitrary small in (5.1) and
can further assume that $g^i_{t_i}$ satisfies
\begin{align} 2\epsilon_0\ge\|g^i_{t_i}-g_{KE}\|_{C^{3}(M)}\ge
\epsilon_0.\end{align}
In particular, curvatures $\text{Rm}(g^i_{t_i})$ are uniformly
bounded. Thus by the regularity of Ricci flow [Sh],
$\|\text{Rm}(g^i_{t_i})\|_{C^{k-2,\frac{1}{2}}(M, g_{ t_i})}$ are
uniformly bounded for any $k\ge 2$. Hence by Cheeger-Gromov's
compactness theorem (cf. [GW]), one sees that there exist a
subsequence $g^{i_k}_{t_{i_k}}$ of $g^i_{t_i}$ and a sequence of
diffeomorphisms $\Phi_{i_k}$ on $M$ such that $\Phi_{i_k}^\star
g^{i_k}_{t_{i_k}}$ converge to a limit Riemannian metric $g_\infty$
on $M$ in the $C^{k,\frac{1}{2}}$-topology. Clearly by (5.4), it
holds
\begin{align} 2\epsilon_0\ge \|g_\infty-g_{KE}\|_{C^{3}(M)}\ge
\epsilon_0.\end{align} On the other hand, again by (5.4), the
sequence of induced complex structures $(\Phi_{i_k}^{-1})_\star\cdot
J\cdot (\Phi_{i_k})_\star$ (maybe after taking a subsequence of
$i_k$ ) will converge to a $C^{k,\frac{1}{2}}$-differential complex
structure $J_\infty$ on $M$. By the regularity, $J_\infty$ is
analytic. Moreover, one can show that $g_\infty$ is K\"ahler with
respect to $J_\infty$.

By the monotonicity of $\lambda(\cdot)$ along the flow (see (3.8) in
Section 3), we have
$$\lambda(g^{i_k})\le \lambda(g^{i_k}_{t_{i_k}}).$$
It follows
$$\lambda(g_\infty)\ge \lim_{i_k} \lambda(g^{i_k})\ge\lambda(g_{KE}).$$
Thus by Lemma 3.4, we see that $g_\infty$ is in fact a
K\"ahler-Einstein metric on $(M,J_\infty)$. Now  by Theorem 5.1,
$(g_\infty, J_\infty)$ is conjugate to $(g_{KE}, J)$ by a
diffeomorphism on $M$. But this contradicts to (5.5). The
contradiction implies that (5.1) is true.

 By (5.1), we have
\begin{align}\|g_t-g_{KE}\|_{C^{3}(M)}\le \epsilon_0\le 1,~\forall~
t>0. \end{align}
Thus,  there exists a
sequence $\{g_{t_i}\}$ of $g_t$ which
converges to a limit Riemannian metric $g_\infty$ on $M$ in
the $C^{3}$-topology. Hence by Lemma 3.5, $g_\infty$ is a
K\"ahler-Ricci soliton.  Moreover, by Corollary 4.6, $\lambda(g_\infty)=(2\pi)^{-n}nV=\lambda(g_{KE})$
 since the K-energy is bounded from below [BM], [DT],  and consequently,
by Lemma 3.4,  $g_\infty$ must be a K\"ahler-Einstein metric. Therefore, by Theorem 5.1, we conclude that
 $g_\infty$ is conjugate to $g_{KE}$ by a diffeomorphism  on M. Finally, The uniqueness of limit  $g_\infty$ enable us to prove
  that the KR-flow $(g_t; g)$ globally  converges to $g_{KE}$. Proposition \ref{prop-2.1} is proved.

\end{proof}

\section{Proof of Theorem \ref{th-3}}

\begin{proof}[Proof of Theorem 2.3]  By Corollary 2.2, there exists a $\tau_0\le 1$ with
   $[0,\tau_0)\subset I$. We need to show that $\tau_0\in I$.
In fact we want to prove that for any
$\delta>0$ there exists a large $T$ such that
\begin{align}\|g_t^s-g_{KE}\|_{C^{3}(M)}\le \delta,~\forall~
t\ge T~~\text{and}~s< \tau_0. \end{align}
As we did for Proposition \ref{prop-2.1}, we will prove it by contradiction. Suppose that we can find a sequence of
K\"ahler metrics $g_{t_i}^{s_i}$, where $s_i\to \tau_0$ and $t_i\to\infty$, of KR-flows
$(g_t^{s_i}; g^{s_i})$ and a sequence of diffeomorphisms $\Psi_i$ on $M$
such that
\begin{align} |\Psi_i^\star
g_{t_i}^{s_i}-g_{KE}|_{C^{3}(M)}\ge  \delta_0>0,\end{align}
for some constant $\delta_0$. Since the KR-flows $(g_t^{s_i}; g^{s_i})$ converge to $g_{KE}$, we may further
assume that $g_{t_i}^{s_i}$ satisfy
\begin{align} |\Psi_i^\star
g_{t_i}^{s_i}-g_{KE}|_{C^{3}(M)}\le  2\delta_0.
\end{align}
Then it follows from Cheeger-Gromov's compactness theorem that there is a subsequence $g^{s_{i_k}}_{t_{i_k}}$
of $g^i_{t_i}$ converging to a K\"ahler metric $(g_\infty, J_\infty)$
with property
\begin{align} 2\delta_0\ge \|g_\infty-g_{KE}\|_{C^{3}(M)}\ge
\delta_0.\end{align}

We claim that $(g_\infty, J_\infty)$ is a K\"ahler-Einstein metric. By Lemma 3.4, we only need to show that $\lambda(g_\infty)=\lambda(g_{KE})=(2\pi)^{-n}nV$.
It follows from Corollary 4.6 and the monotonicity of $\lambda(g_t^{\tau_0})$ that
for any $\epsilon>0$, there exists $T> 0$ such that
$$\lambda(g^{\tau_0}_t)\ge (2\pi)^{-n}nV-\frac{\epsilon}{2},~~\forall~t\ge T.$$
Since K\"ahler-Ricci flow is stable in finite time and $\lambda(g_t^s)$ is monotonic in $t$, there is a small $\delta> 0$ such that
for any $s\ge \tau_0-\delta$, we have
\begin{align}\lambda(g_t^s)\ge (2\pi)^{-n}nV-\epsilon,~~\forall~t\ge T.\end{align}
Since $s_i\to \tau_0$ and $t_i\to \infty$, we conclude that
$$\lim_{s_i\to \tau_0, t_i\to \infty}\lambda(g^{s_i}_{t_i})=(2\pi)^{-n}nV.$$
By the continuity of $\lambda(\cdot)$, it follows
\begin{align}\lambda(g_\infty)=(2\pi)^{-n}nV=\lambda(g_{KE}).\end{align}
This proves the claim.

Now  by Theorem 5.1,  we see that $(g_\infty, J_\infty)$ is
conjugate to $(g_{KE}, J)$ by a diffeomorphism on $M$. But this
contradicts to (6.4). This contradiction implies that (6.1) is true.

 By (6.1), we have
\begin{align}\|g_t^{\tau_0}-g_{KE}\|_{C^{3}(M)}\le \delta_0\le 1,~\forall~
t\ge T. \end{align}
Thus,  there exists a
sequence $\{g^{\tau_0}_{t_i}\}$ of $g_t^{\tau_0}$ which
converges to a limit Riemannian metric $g_\infty$ on $M$ in the $C^{3}$-topology. Hence, by Lemma 3.5, $g_\infty$ is a
K\"ahler-Ricci soliton.  Moreover, by Corollary 4.6, $\lambda(g_\infty)=(2\pi)^{-n}nV$, so $g_\infty$ must be a K\"ahler-Einstein metric.
Therefore,  by Theorem 5.1, we conclude that
 $g_\infty$ is conjugate to $g_{KE}$ by a diffeomorphism on M. Finally, the uniqueness of limit  $g_\infty$ enable us to prove
 that the KR-flow $(g_t^{\tau_0}; g^{\tau_0})$ globally  converges to $g_{KE}$.  This shows  $\tau_0 \in I$.

\end{proof}

\begin{proof}[Completion of proof of Theorem 1.1] We have shown that $I=[0,1]$, so the first part of Theorem 1.1 is proved.
The second part of Theorem 1.1 is known (cf. [PS], also see [CT2]). For the reader's convenience, we outline its proof.
Since the limiting complex structure $J_\infty$ is conjugate to $J$, by an estimate (3.11) in [PS],
one sees that there exists two uniform constants $C,\delta>$ such that
\begin{align}\int_M|\nabla h_t|_{g_t}^2\omega_{g_t}^n\le Ce^{-\delta t}, ~~\forall ~t>0.\end{align}
Moreover, by using (6.8), one can further derive
$$\int_M|\nabla^p h_t|_{g_t}^{2}\omega_{g_t}^n\le C_pe^{-\delta_p t}, ~\forall ~t>0,
$$
where $p>0$ is any integer and two uniform constants
$C_p,\delta_p>0$ depend only on $p$ (cf. [CT2], [PS], [Zhu],
etc). By using the Sobolev embedding theorem, it follows
$$\| h_t-a\|_{C^0(M)}\le C'e^{-\delta' t},~~\forall ~t>0,$$
for some constant $a$. As a consequence, we get
$$\|\phi_t-\frac{1}{V}\int_M \phi_t \omega_{g_{KE}}^n\|_{C^0(M)}\le C''.$$
Using this, we can prove that all $C^\ell$-norms of $\phi_t$ decay exponentially.
Hence, $\phi_t$ are convergent
and the corresponding metrics $\omega_{\phi_t}$ converge exponentially to a K\"ahler-Einstein metric
$g_{KE}'$ on $(M,J)$. By the uniqueness theorem on K\"ahler-Einstein metrics, $g_{KE}'$ must be conjugate to $g_{KE}$ by a holomorphism transformation.
The proof of Theorem 1.1 is completed.

\end{proof}

\section{Proof of Theorem  \ref{th-2}}

As in the proof of Theorem \ref{th-1},  we consider a path of
K\"ahler forms
$\omega_s=\omega_{g_0}+s\sqrt{-1}\partial\overline\partial\phi$ with
$\omega_g=\omega_{g_0}+\sqrt{-1}\partial\overline\partial\phi$ and
set
\begin{align} I=\{& s\in ~[0,1] |~~ \text{KR-flow} ~(g_t^s;\omega_s)~\text {
converges
to}~ g_{KE} ~\text{ in the}~C^\infty\text{-topology} \}.\notag
 \end{align}
By the assumption, $0\in I$. We shall show that $I$ is both open and closed.

The following lemma should be due to Chen [Ch]. It generalizes Bando-Mabuchi's result about the lower bound of K-energy [BM].

\begin{lem} If there exists a sequence of K\"ahler metrics on $(M,J)$  with
K\"ahler class $2\pi c_1(M)$ converging to a
K\"ahler-Einstein metric $g_{KE}$ for some complex structure $J_0$
in the $C^\infty$-topology, then Mabuchi's K-energy is bounded from
below on the space of K\"ahler potentials on $(M,J)$.
\end{lem}

\begin{proof} By our assumption, there are diffeomorphisms $\Psi_i$ of $M$ such that $|(\Psi_i^{-1})_\star\cdot J \cdot (\Psi_i)_\star- J_0|_{C^k(M)}$ converge
 to $0$ ( $\forall~ k>0$). Then an observation of Szekylihidi in [Sz] implies
 that there exists a smooth holomorphic fibration: $\pi: Y\mapsto D$, where $D$ denotes the unit disc in the complex plane,
such that 1) $\pi^{-1}(z)$ is boholomorphic to $(M,J)$; 2)
$\pi^{-1}(0)$ is a compact K\"ahler manifold $(M,J')$ which admits a
K\"ahler-Einstein metric $g_{KE}'$. In fact, by [CS],
 this $J'$ is conjugate to $J_0$. Thus the K-energy is bounded
from below on space of K\"ahler potentials on $(M,J)$ by Chen's main
result in [Ch].
\end{proof}

The following proposition shows that $I$ is an open set.

\begin{prop}  \label{prop-7.3} Let  $(M,J)$ be  a compact K\"ahler manifold with $c_1(M)>0$ and $g_{KE}$
be a K\"ahler-Einstein  metric with K\"ahler class $2\pi c_1(M)$
for some complex structure $J'$ on $M$.  Suppose that there exists a
K\"ahler metric $(J,g_0)$ with  K\"ahler class $2\pi c_1(M)$ such
that the KR-flow with the initial metric $g_0$ converges to $g_{KE}$
in the $C^{3}$-topology. Then there exists $\delta>0$
depending only on $g_{KE}$ and $g_0$ such that for any K\"ahler
metric $g$ on $(M,J)$ with K\"ahler class $2\pi c_1(M)$ satisfying
 \begin{align} \|g-g_0\|_{C^{3}(M)}\le \delta,
 \end{align}
 the KR-flow $(g_t;g)$ converges to $g_{KE}$ in the $C^\infty$-topology.
 \end{prop}

\begin{proof} Its proof is identical to that for Proposition 2.1.
Since the KR-flow is stable in finite time, by our assumption, we see that for any $\epsilon>0$, there exist $\delta> 0$
and a large $T$ depending only on $\epsilon$ and $g_0$ such that for any K\"ahler metric $g$
on  $(M,J)$ with K\"ahler class $2\pi c_1(M)$, whenever
$$\|g-g_0\|_{C^{3}(M)}\le \delta,$$
then the evolved K\"ahler metric $g_T$ of KR-flow $(g_t; g)$ satisfies
\begin{align}\| g_T-g_{KE}\|_{C^{3}(M)}< \epsilon.\end{align}
We claim that (7.2) holds for any $t\ge T$. If the claim is false,
there is a sequence of K\"ahler metrics
$g^i$ on $(M,J)$ with K\"ahler class $2\pi c_1(M)$ satisfying
\begin{align} ||g^i-g_0||_{C^{3}(M)}\to 0, ~\text{as}~ i\to\infty,\end{align} such
that there exist $t_i\to\infty$ satisfying:
\begin{align}  2 \epsilon \ge\|g^i_{t_i}-g_{KE}\|_{C^{3}(M)}\ge
\epsilon,\end{align} where $g^i_{t_i}$ is the evolved metric of the
KR-flow $(g^i_t; g^i)$ at time $t_i$. Then by Cheeger-Gromov's
compactness theorem, there exists a subsequence $g^{i_k}_{t_{i_k}}$
of $g^i_{t_i}$ which converges to a K\"ahler metric $(g_\infty,
J_\infty)$ with property
\begin{align} 2\epsilon\ge \|g_\infty-g_{KE}\|_{C^{3}(M)}\ge
\epsilon.\end{align} Moreover, by the monotonicity of
$\lambda(\cdot)$ and (7.3), it is easy to see
$$\lambda(g_\infty)\ge\lambda(g_{KE}).$$
It follows from Lemma 3.4 that $g_\infty$ is actually a
K\"ahler-Einstein metric on $(M,J_\infty)$. Consequently, by Theorem
5.1, $(g_\infty, J_\infty)$ is conjugate to $(g_{KE}, J')$ by a
diffeomorphism on $M$. But this contradicts to (7.5). So the claim
is true.

By the claim, we see that
\begin{align}\|g_t-g_{KE}\|_{C^{3}(M)}\le C,
\end{align}
for all K\"ahler metrics of $(g_t; g)$ as long as $g$ satisfies (7.1). Then there exists a
sequence $\{g_{t_i}\}$ of $g_t$ which converges to a K\"ahler-Ricci soliton $g_\infty$ according to Lemma 3.5.  Moreover, since the K-energy is bounded from below according to Lemma 7.1, one can easily show that $g_\infty$ is actually a K\"ahler-Einstein metric, which has to be conjugate to $g_{KE}$ by a diffeomorphism of M.
Therefore, the flow $(g_t; g)$
converges to $g_{KE}$ in the $C^\infty$-topology.

 \end{proof}

\begin{proof}[Proof of Theorem 1.2] By Proposition \ref{prop-7.3}, we only need to prove that $I$ is closed. However, the proof
of Theorem 2.3 still applies here since Corollary 4.6 is still true
by Lemma 7.1. So we refer the readers to Section 6 for showing that
$I$ is closed.

\end{proof}

\section{Appendix  }

In this appendix, we give a uniform $L^\infty$-estimate for
minimizing solutions of (3.5) associated to the evolved K\"ahler
metrics of KR-flow. The result was used in Section 4 and is of independent interest.

\begin{prop} Let $f=f_t$ be a minimizer of $W(g_t,\cdot)$ for the evolved K\"ahler metrics $g=g_t$ of (2.1). Then $f$ is uniformly bounded.
\end{prop}

\begin{proof} Let $u=e^{-\frac{f}{2}}$. Then by (3.5), $u$ satisfies
\begin{align}\Delta u-\frac{1}{2}fu-\frac{1}{4}Ru=\frac{1}{2}\lambda(g)u.\end{align}
Since the scalar curvature $R$ is uniformly bounded by Lemma 4.1,  we
have
\begin{align} -\Delta u\le u^{1+\delta}+C_\delta,\end{align}
for any $\delta>0$. Thus for any $p>1$, we have
$$-u^{\frac{p}{2}}\Delta u^{\frac{p}{2}}\le
\frac{p}{2}u^{p+\delta}+C_\delta\frac{p}{2}u^{p-1+\delta}.$$ By the
integration by parts, we get
\begin{align} \int_M|\nabla u^{\frac{p}{2}}|^2 dV_g&\le p (\int_M u^{p+\delta}dV_g+C_1)\notag\\
&\le
p[(\|u\|_{L^{\frac{(1+\epsilon)\delta}{\epsilon}}})^{\frac{\epsilon}{1+\epsilon}})
(\|u\|_{L^{(1+\epsilon)p}})^{\frac{1}{1+\epsilon}}+ C_1]\notag\\
&\le
C_2p[(\|u\|_{L^{(1+\epsilon)p}})^{\frac{1}{1+\epsilon}}+1],\end{align}
where $\epsilon<\frac{2}{n-2}$. At the last inequality we used the
fact $\int_M u^2=\int_M dV_g=V$. On the other hand, since the scalar curvature $R$ is uniformly bounded, we have the Sobolev inequality (cf. [Zha]):
$$ \int_M |\nabla \psi|^2 dV_g \ge \delta_0 \int_M \psi^{\frac{2n}{n-2}} dV_g- C_0\int_M \psi^2 dV_g,$$
for any $W^{1,2}$-function $\psi$, where  $\delta_0$ and $C_0$ are uniform constants. In particular, for
$\psi=u^{\frac{p}{2}}$, we have
$$\int_M|\nabla u^{\frac{p}{2}}|^2dV_g\ge \delta_0(\|u\|_{L^{\frac{np}{n-2}}})^{\frac{n-2}{n}} -C_0'\int_M u^pdV_g.$$
Thus we derive from (8.3)
$$\|u\|_{L^{\frac{np}{n-2}}}\le(C_3p)^{\frac{1}{p}}
(\|u\|_{L^{(1+\epsilon)p}}+1).$$
Iterating the above over $p$, we can deduce
$$\|u\|_{L^\infty}\le C_4(\int_M u^2dV_g+1)=C_4(V+1),$$
consequently,
\begin{align} f\ge -C'.\end{align}

Note the normalized condition $\int_M  e^{-f}dV_g=V$. By (8.4),
we can find an $A$ and an open set $E_A$ of $M$ for each $f$ such that
$$\int_{E_A} e^hdV_g \,> \,\delta,$$
where $E_A\subset\{x\in~ M|~~f(x) < A\}$ and $h=h_t$ is the Ricci potential function of
$g=g_t$ with the normalization
$$\int_Me^hdV_g=V.$$
Note that $h$ is uniformly bounded by Lemma 4.1. Hence, we have
\begin{align}\int_M fe^hdV_g &=\int_{M\setminus E_A}
 fe^hdV_g+\int_{E_A}
fe^hdV_g\notag\\
&\le \left(\int_{M\setminus E_A} e^h dV_g \right)^{\frac{1}{2}}\left(\int_{M\setminus E_A}
f^2e^hdV_g\right)^{\frac{1}{2}} + C A.\notag\end{align}
It follows
\begin{align}\int_M f^2e^hdV_g \ge \frac{1}{V-\delta}(\int_M
fe^hdV_g)^2-a (\int_M f^2e^hdV_g)^{\frac{1}{2}}-C',\end{align}
for some constant $a,C'>0$.  On the other hand, by integrating (3.5) at both two sides, we have
$$ \int_M |\nabla f|^2dV_g\le \int_M fdV_g+C''.$$
Then by using the following weighted Poincare's
inequality (cf. [TZ3]),
$$\int_M |\nabla f|^2e^hdV_g\ge  \int_M f^2e^hdV_g -\frac{1}{V}(\int_M
fe^hdV_g)^2,$$
we obtain
\begin{align}\int_M f^2e^hdV_g \le \frac{1}{V}(\int_M fe^hdV_g)^2+
b\int_M fdV_g+C''',\end{align}
for some constant $b,C'''>0$. Combining this with (8.5), we deduce
$$\int_M f^2 e^hdV_g \le C_6, $$
consequently,
\begin{align}\int_M f^2dV_g \le C_7. \end{align}

It follows from (3.5) that $f$ satisfies,
$$\Delta f\ge - f -B.$$
Then by the standard $L^\infty$-estimate for
elliptic equations of second order, we have
$$\|f^+\|_{L^\infty}\le C(\int_M f^2dV_g+1)\le C_8,$$
where $f^+=\text{max}\{f,0\}$. Thus we have proved that $f$ is uniformly bounded.
\end{proof}

\enddocument

\end{document}